\documentclass[12pt]{amsart}

\gdef\private#1{#1} \gdef\public#1{}  

\usepackage{amssymb}
\usepackage{xspace}
\usepackage{enumerate}
\usepackage{verbatim}
\usepackage{ifthen}
\usepackage[plain]{fancyref}
\usepackage{hyperref}
\usepackage[normalem]{ulem}
\private{\usepackage[utf8]{inputenc}}

\newtheorem{thm}{Theorem}
\newtheorem{lem}[thm]{Lemma}
\newtheorem{prop}[thm]{Proposition}

\newtheorem*{claim}{Claim}

\theoremstyle{definition}
\newtheorem{defn}{Definition}

\theoremstyle{remark}

\newtheorem{rem}[thm]{Remark}

\newcommand*\fancyrefthmlabelprefix{thm}\frefformat{plain}{\fancyrefthmlabelprefix}{Theorem~#1}
\newcommand*\fancyreflemlabelprefix{lem}\frefformat{plain}{\fancyreflemlabelprefix}{Lemma~#1}
\newcommand*\fancyrefproplabelprefix{prop}\frefformat{plain}{\fancyrefproplabelprefix}{Proposition~#1}
\newcommand*\fancyrefcorlabelprefix{cor}\frefformat{plain}{\fancyrefcorlabelprefix}{Corollary~#1}
\newcommand*\fancyrefclaimlabelprefix{claim}\frefformat{plain}{\fancyrefclaimlabelprefix}{Claim~#1}
\newcommand*\fancyreffactlabelprefix{fact}\frefformat{plain}{\fancyreffactlabelprefix}{Fact~#1}
\newcommand*\fancyrefquestionlabelprefix{question}\frefformat{plain}{\fancyrefquestionlabelprefix}{Question~#1}
\newcommand*\fancyrefconjlabelprefix{conj}\frefformat{plain}{\fancyrefconjlabelprefix}{Conjecture~#1}
\newcommand*\fancyrefdefnlabelprefix{defn}\frefformat{plain}{\fancyrefdefnlabelprefix}{Definition~#1}
\newcommand*\fancyrefconstlabelprefix{const}\frefformat{plain}{\fancyrefconstlabelprefix}{Construction~#1}
\newcommand*\fancyrefsetuplabelprefix{setup}\frefformat{plain}{\fancyrefsetuplabelprefix}{Setup~#1}
\newcommand*\fancyrefexlabelprefix{ex}\frefformat{plain}{\fancyrefexlabelprefix}{Example~#1}
\newcommand*\fancyrefremlabelprefix{rem}\frefformat{plain}{\fancyrefremlabelprefix}{Remark~#1}
\frefformat{plain}{\fancyrefeqlabelprefix}{(#1)}
\newcommand*\fancyrefitemlabelprefix{item}\frefformat{plain}{\fancyrefitemlabelprefix}{(#1)}

\def\repeat#1#2 {\expandafter\gdef\csname B#1\endcsname {\mathbb{#1}}
  \ifthenelse{\equal{#2}{*}}{}{\repeat #2 }}
\repeat ABCDEFGHIJKLMNOPQRSTUVWXYZ*
\def\repeat#1#2 {\expandafter\gdef\csname C#1\endcsname {\mathcal{#1}}
  \ifthenelse{\equal{#2}{*}}{}{\repeat #2 }}
\repeat ABCDEFGHIJKLMNOPQRSTUVWXYZ*

\let\isom\cong
\let\normal\triangleleft
\let\phi\varphi

\DeclareMathOperator{\Sym}{Sym}
\DeclareMathOperator{\Alt}{Alt}

\DeclareMathOperator{\Comp}{Comp}

\DeclareMathOperator{\Soc}{Soc}
\DeclareMathOperator{\Pker}{P\kern-.3pt ker}

\newcommand{\xxx}[2][] {%
  \private{\ifthenelse{\equal{#1}{}}
    {\underline{$\bullet$}}{\uline{#1}}\marginpar{\tiny #2}}%
  \public{#1}\xspace}

\begin{document}
\title[A CFSG-free Analysis]{A CFSG-free analysis of Babai's quasipolynomial GI algorithm}
\author{L\'aszl\'o Pyber}
\thanks{The author was supported by the National Research, Development and Innovation Office (NKFIH) Grants K115799, ERC\_HU\_15 118286.}
\maketitle

Babai's analysis of his algorithm \cite{Ba1}
(corrected in \cite{Ba2})
uses
the Classification of Finite Simple Groups (CFSG) in Section 8.
In particular, he uses Schreier's Hypothesis
(a well-known consequence of CFSG),
which states that the outer automorphism groups of finite simple
groups are solvable.
We will make Babai's bound worse, but still quasipolynomial, and
remove CFSG entirely from the analysis.
Our short proof relies on a little-known result of Wielandt
\cite{Hig}.

More precisely, we give a CFSG-free proof of a weaker version of
the key group-theoretic result of Babai,
Lemma 8.5 from \cite{Ba1}, see \fref{lem:key}.
It is acknowledged in \cite{Ba3} that our result is sufficient to make
Babai's work CFSG-free.

\vskip 5pt
All groups in this note are finite. Throughout $\log$ means logarithm
to the base 2.
We use the following notation:
\begin{itemize}
\item
  $|G|$ -the order of a group $G$;
\item
  $\Sym(\Omega)$ -the symmetric group on the set $\Omega$; 
\item
  $\Alt(n)$ -the alternating group of degree $n$;
\item
  For $G < \Sym(\Omega)$, $G_x$ denotes the pointstabiliser of a point
  $x\in\Omega$ ($G$ will always be transitive and  all pointstabilisers
  equivalent as permutation groups);
\item
  If $G < \Sym(\Omega)$, and 
  $\Delta$ is a subset of $\Omega$ fixed by $G$,
  then the image of the action of $G$ on $\Delta$ is called the \emph{permutation group induced on} $\Delta$ by $G$, and written $G^\Delta$;

\item
  $\Comp_A(G)$ -the set of isomorphism types of abelian composition factors
  of a group $G$.
  Note that we have $\big|\Comp_A(G)\big|\le\log|G|$.
\end{itemize}

For notation and definitions unexplained here and below see
the monograph of Dixon and Mortimer \cite{DM}.

We will first prove a useful result for groups of odd order
(without using even the Feit-Thompson theorem).
We call the orbits of a pointstabiliser in a transitive group
\emph{suborbits}.
We need the following basic lemma of Jordan (see \cite[18.2]{Wie}):

\begin{lem} \label{lem:Jordan}
  Let $G$ be a primitive group of degree $n$
  and $G_x$ a pointstabiliser.
  If a prime $p$ divides $|G_x|$ then $p$ divides
  $\big|{G_x}^\Delta\big|$
  for all suborbits $\Delta$ of length at least $2$.
\end{lem}

Recall that the \emph{exponent} of a group is the smallest common
multiple of the orders of its elements.

\begin{lem}\label{lem:nulladik}
  Let $G$ be a transitive group of degree $n$.
  If $G$ has odd order then the exponent of $G$ is at most $n^{(\log n)^2}$.
\end{lem}

\begin{proof}
  We will actually prove the following claim:
  \begin{claim}
    The product $\pi(G)$ of all different primes
    dividing $|G|$ is at most $n^{\log n}$.
  \end{claim}

  Since any element of $G$ of prime power order has order at most $n$,
  the claim implies our statement.

  Assume first that $G$ is primitive.
  Since $G$ has odd order, $G$ cannot be doubly transitive.
  The smallest non-trivial suborbit $\Delta$ has size $<\frac n2$.
  \fref{lem:Jordan} implies that
  $\pi(G_x)=\pi\big({G_x}^\Delta\big)$.
  It is clear that
  $\pi(G)$ divides $n\cdot\pi(G_x)$,
  hence by induction
  $\pi(G)\le n\cdot\left(\frac n2\right)^{\log\frac n2}\le n^{\log n}$.

  Assume now that $G$ is imprimitive.
  Let $B_1$ be a minimal non-trivial block of imprimitivity (of size $b$)
  and $B_1,\dots,B_k$ the corresponding system of blocks of
  imprimitivity.
  Then $G$ permutes these blocks transitively, let $K$ denote the
  kernel of this action.
  By induction $\pi\big(G/K\big)\le k^{\log k}$.

  We may assume that $K\neq1$.
  Then the minimality of $B_1$ implies that
  $K$ acts transitively on each $B_i$,
  and the corresponding transitive groups are permutation equivalent.
  Hence $\pi(K)=\pi\big(K^{B_1}\big)\le b^{\log b}$.

  Using $kb=n$ we obtain that $\pi(G)\le n^{\log n}$.
  This completes the proof.
\end{proof}

We now prove an easy result which
implies that to obtain a polylogarithmic
bound for $\big|\Comp_A(G)\big|$
in terms of the degree of the transitive group $G$ it is sufficient to
obtain such a bound  for $\big|\Comp_A(G_x)\big|$
(where $G_x$ is a pointstabiliser).

\begin{prop}\label{prop:3}
  Let $G < \Sym(\Omega)$ be a transitive group of degree $n$ and $G_x$
  a pointstabiliser. Then
  $$
  \big|\Comp_A(G) \setminus \Comp_A(G_x)\big| <\log n.
  $$
\end{prop}
\begin{proof}
  Let $G=G_0\triangleright G_1\triangleright G_2\triangleright\dots$
  be a composition chain of $G$.
  Then the subgroups $G_i\cap G_x$ form a subnormal chain
  $G_x\trianglerighteq G_1\cap G_x\trianglerighteq G_2\cap G_x\dots$
  (possibly with repetitions).
  Now $G\big/ G_1\ge G_1\cdot G_x\big/G_1\isom G_x\big/(G_1\cap G_x)$
  and similarly
  $G_i\big/ G_{i+1}\ge G_{i+1}\cdot(G_i\cap G_x)\big/G_{i+1}\isom
  (G_i\cap G_x)\big/(G_{i+1}\cap G_x)$
  for $i\ge1$.

  If $p$ is a prime not dividing $n$,
  then the product of the $p$-parts of the orders
  $\big|G_i\big/G_{i+1}\big|$
  must be the same as the product of the $p$-parts of the orders
  $\big|(G_i\cap G_x)\big/(G_{i+1}\cap G_x)\big|$.
  This implies that if $G$ has a composition factor of order $p$ then
  $G_x$ also has such a composition factor.
  Our statement follows.

\end{proof}

The following result of Wielandt, that appeared first in \cite{Hig} (which
the author found in the insightful survey of Cameron \cite{Cam})
explains why we have
to concentrate on pointstabilisers of transitive groups, rather than on transitive
groups themselves (see \cite[p~110--111]{Cam} for a more detailed explanation).

Since Wielandt's beautiful result is at the heart of our argument,
we will reproduce here his short proof from \cite{Hig}.

For this we need some more notation.

The orbits of $G\leq \Sym(\Omega)$ on $\Omega\times\Omega$ are called
the \emph{orbitals} of $G$ on $\Omega$. For each orbital $\Delta$
there is a \emph{paired orbital} denoted $\Delta'$, where
$(y,x)\in\Delta'$ if and only if $(x,y)\in\Delta$. For each orbital
$\Delta$ of $G$ and each $x\in\Omega$ we define
$\Delta(x)=\left\{y\in\Omega|(x,y)\in\Delta\right\}$.
Such sets are exactly the suborbits of $G_x$. If $\Delta$ and
$\Delta'$ are paired orbitals, then $\Delta(x)$ and $\Delta'(x)$ are
called paired suborbits. A suborbit $\Delta(x)$ is called
\emph{self-paired} if $\Delta(x)=\Delta'(x)$. (To simplify notation we
often denote $\Delta(x)$ by $\Delta$.) 

\begin{lem}\label{lem:4} \textbf{\emph{(Wielandt)}}
  Let $G < \Sym(\Omega)$ be a nonregular primitive group and let
  $\Delta(x)$ be a $G_x$-orbit $\ne\{x\}$,
  let $y$ be an element of $\Delta(x)$ and
  let $y'$ be an element of $\Delta'(x)$,
  where $\Delta'(x)$ is the $G_x$-orbit
  paired with $\Delta(x)$.
  Let $T(x)$ be the kernel of the action of $G_x$ on
  $\Delta(x)$.
  Then all composition factors of $T(x)$ appear as
  composition factors of
  $(G_{x,y})^{\Delta(x)}$ or of $(G_{x,y'})^{\Delta'(x)}$.
\end{lem}
\begin{proof}
  For a subgroup $H$ of $G$, denote by $H^*$ the smallest subnormal
  subgroup of $H$ such that
  every composition factor between $H$ and $H^*$ is a
  composition factor of
  $G^{\Delta(x)}_{x,y}$
  or of
  $G^{\Delta'(x)}_{x,y'}$;
  $H^*$ is a characteristic subgroup of $H$
  (Wielandt \cite[Th. 13, p. 220]{Wie2}).
  Now   $G^{\Delta(x)}_{x,y}\isom G_{x,y}\big/T(x)$
  and therefore $G_{x,y}^*=T(x)^*$.
  Similarly $G_{x,y'}^*\isom U(x)^*$,
  where $U(x)$ denotes the pointwise stabiliser of
  $\{x\}\cup\Delta'(x)$.
  We can choose the notation so that $\Delta(x)^g=\Delta(x^g)$
  for all $x\in\Omega$, $g\in G$.
  Then $\Delta'(x)^g=\Delta'(x^g)$
  and $y\in\Delta(x)$ implies $x\in\Delta'(y)$
  so $G_{x,y}^*=U(y)^*$.
  Hence
  $T(x)^* = U(y)^*\normal \langle G_x, G_y\rangle = G$,
  so that $T(x)^* = 1$
  (since a non-trivial normal subgroup of a primitive group
  must be transitive),
  and the theorem is proved.
\end{proof}

It was a great surprise to the present author that one can prove a
useful result for pointstabilisers of arbitrary transitive groups by
an elementary induction argument
(for transitive groups themselves using induction is quite standard).
However a paper of Isaacs \cite{Is}, where abelian pointstabilisers in
transitive groups are shown to have a distinctive numerical property
gave a strong indication that such an argument may exist and may not
rely on CFSG. (The theorem of Isaacs in \cite{Is} in turn may be traced
back to some results and questions in \cite{BGPy}).
Note also that \fref{thm:7}
below is not very hard to prove
using well-known consequences of CFSG.

\begin{defn}
  Let $D$ be a direct product of the groups $D_1,D_2,\dots,D_t$.
  A subdirect product subgroup  of $D$ is a subgroup $G$
  which projects onto all the constituents $D_1,D_2,\dots,D_t$.
\end{defn}

The following well-known result (see \cite[Ch.2 (4.19)]{Su})
describes the structure of subdirect product
subgroups of two groups.

\begin{lem} \label{lem:subdirect-product-structure}
  Let $G$ be a subdirect product subgroup of $D=D_1\times D_2$.
  Then setting $N_1=D_1\cap G$ and $N_2=D_2\cap G$
  we have
  $N_1\times N_2\normal G$, and
  $G\big/(N_1\times N_2)\isom D_1\big/N_1\isom D_2\big/N_2$.
\end{lem}

\begin{prop}\label{prop:5}
  The composition factors of $G$ are among
  the composition factors of $D_1,D_2,\dots,D_t$.
\end{prop}
\begin{proof}
  Easy induction using \fref{lem:subdirect-product-structure}.
\end{proof}

The key result of this note is the following.

\begin{thm}\label{thm:7}
  Let $G$ be a transitive permutation group of degree
  $n$ and $G_x$ a pointstabiliser. Then
  $\big|\Comp_A(G_x)\big| < 2(\log n)^2$.
\end{thm}

The proof of \fref{thm:7} is based on two reduction lemmas:
\fref{lem:reduct-trans-groups} and \fref{lem:reduct-prim-groups}.

\begin{defn}
  A group $G < \Sym(\Omega)$ is called \emph{quasiprimitive} if it is
  transitive and each of its non-trivial normal subgroups are also
  transitive. Primitive permutation groups are quasiprimitive
  (see \cite{Pr} for more on quasiprimitive qroups).
\end{defn}

\begin{lem}[Reduction for transitive groups]
  \label{lem:reduct-trans-groups}
  Let $G$ be a transitive imprimitive group of degree $n$.
  We have
  $$
  \big|\Comp_A(G_x)\big| \le
  \big|\Comp_A(X_\alpha)\big|+\big|\Comp_A(Y_\beta)\big|+\log n,
  $$
  where $X$ and $Y$ are certain transitive groups acting on sets
  of size $t\ge2$ and $m\ge2$ respectively, where $tm\le n$,
  and $X_\alpha$ and $Y_\beta$ are pointstabilisers.
\end{lem}
\begin{proof}
\begin{enumerate}[\indent a)]
\item
  Let us first assume that $G$ is quasiprimitive but not primitive.
  Since $G$ is not primitive $G_x$ is not a maximal subgroup of $G$.
  Consider a proper maximal subgroup $M$ of $G$ containing $G_x$.
  Let $N$ be the (unique) maximal normal subgroup of $G$ contained in $M$.
  If $N\ne1$, then
  the quasiprimitivity of $G$ implies that $G=NG_x\le M$, a contradiction.
  Hence the (primitive) representation of $G$ on the cosets of $M$
  is faithful of degree, say, $t$.

  Consider now the representation of $M$ on the cosets of $G_x$.
  The kernel of this action is a normal subgroup $K$ of $G_x$,
  hence all composition factors
  of $K$ appear among the composition factors of $M$,
  which by the above is the pointstabiliser of a transitive group
  of degree $|G:M| =t$.
  The ``remaining'' composition factors of $G_x$ are the composition factors of
  $G_x\big/K$, which is clearly a pointstabiliser of
  a transitive group of degree $|M:G_x|= n/t$.
\item
  Let us next assume that $G$ is a transitive group which has a non-trivial
  intransitive normal subgroup $N$. The orbits of $N$ form a system of
  imprimitivity $B=B_1,B_2,\dots,B_t$ of $G$.
  Denote by $K$ the kernel of the action of $G$ on this system of
  blocks.
  Denote by $\tilde G$ the image of this action
  and by $\tilde G_{\tilde B}$ the stabiliser of the point
  ${\tilde B}$.
  Clearly this stabiliser has index $t$ in $\tilde G$.
  Let $H$ be the inverse image of this stabiliser
  $\tilde G_{\tilde B}$ in $G$, a subgroup of index $t$ in $G$.
  It is clear that $H$ contains $KG_b$ for an element $b$ of $B$.
  Since $K$ is transitive on $B$
  we have $|B|=|K:K \cap G_b|$.
  But this is equal to $|KG_b:G_b|$ and therefore actually $H=KG_b$.
  Now $H/K=KG_b/K$ is isomorphic to $G_b/(K \cap G_b)$
  hence the composition factors in this quotient of $G_b$
  are the composition factors of $\tilde G_{\tilde B}$.

  The remaining composition factors of $G_b$ are
  the composition factors of $K_b$.
  Let $J$ be the kernel of the action of $K_b$ on $B$.
  Now $J$ is also the kernel of the action of $K$ on $B$
  hence a normal subgroup of $K$.
  Note that $K$ is a subdirect product subgroup of the $t$-th direct power
  of some transitive group $T$ of degree $n/t=m$.

  Now $K_b/J$ is permutation equivalent to
  a pointstabiliser in $T$.
  By \fref{prop:5} the composition factors of $J$ are
  among the composition factors of $T$.
  By \fref{prop:3},
  at most $\log m$ of the composition factors of $T$ (hence $J$)
  do not occur among the composition
  factors of a pointstabiliser in $T$,
  hence the same is true for $K_b$.
\end{enumerate}
\end{proof}

Recall that the socle $\Soc(G)$ of a group $G$ is the product of its
minimal normal subgroups. The following is well-known (see also \cite{Ba1}).

\begin{prop} \label{prop:Socle}
  \begin{enumerate}
  \item \label{item:6}
    The socle is a direct product of simple groups.
  \item \label{item:7}
    The socle of a primitive permutation group is a direct product of
    isomorphic simple groups.
  \item \label{item:8}
    If the socle of a primitive permutation group $G$ is abelian then it is
    elementary abelian of order $p^s=n$ and $G\big/\Soc(G)$ embeds into $GL(s,p)$.
    In particular we have $\big|G\big/\Soc(G)\big|<p^{(s^2)}$.
  \end{enumerate}
\end{prop}

\begin{lem}[Reduction for primitive groups]
  \label{lem:reduct-prim-groups}
  Let $G$ be a primitive group of degree $n$
  with a self-paired suborbit $\Delta$.
  Then one of the following holds:
  \begin{enumerate}
  \item \label{item:1}
    $
    \big|\Comp_A(G_x)\big|\le(\log n)^2
    $
  \item \label{item:2}
    $
    \big|\Comp_A(G_x)\big|\le\big|\Comp_A(P_y)\big|
    $
    for some primitive permutation group $P$ acting on $\Delta$
    and some $y\in\Delta$,
  \item \label{item:3}
    $
    \big|\Comp_A(G_x)\big| \le
    \big|\Comp_A(X_\alpha)\big|+\big|\Comp_A(Y_\beta)\big|+2\log n,
    $
    where $X$ and $Y$ are certain transitive groups acting on sets
    of size $t\ge2$ and $m\ge2$ respectively, where $tm\le|\Delta|\le n$,
    and $X_\alpha$ and $Y_\beta$ are pointstabilisers.
  \end{enumerate}
\end{lem}
\begin{proof}
  We may assume that $G$ is nonregular.
  Let $\Delta$ be the smallest non-trivial self-paired orbit of
  $G_x$.
  Let $P$ denote the image of the action of $G_x$ on $\Delta$.
  By \fref{lem:4}
  $$
  (*)\kern 20pt
  \Comp_A(G_x)\subseteq\Comp_A(P)\cup\Comp_A(P_y).
  $$
  Assume first that $P$ is primitive.
  If $\Soc(P)$ is abelian, then
  using \fref{prop:Socle}\fref{item:8}
  it is easy to see that \fref{item:1} holds.
  If $\Soc(P)$ is non-abelian,
  then $\Comp_A(P)=\Comp_A\big(P\big/\Soc(P)\big)$.
  Clearly $\Soc(P)\cdot P_y=P$, hence
  $$
  P\big/\Soc(P)\isom P_y\big/\big(P_y\cap\Soc(P)\big).
  $$
  In particular, $\Comp_A(P)\subseteq\Comp_A(P_y)$,
  and $(*)$ implies \fref{item:2}.

  Assume next that $P$ is imprimitive.
  \fref{prop:3} and $(*)$ imply
  that
  $$
  \big|\Comp_A(G_x)\big|\le
  \big|\Comp_A(P_y)\big|+\log n.
  $$
  Applying \fref{lem:reduct-trans-groups} to $P$ acting on $\Delta$
  we obtain \fref{item:3}.
\end{proof}

\begin{proof}[The proof of \fref{thm:7}]
  The proof of the theorem
  will be simple induction based on 
  \fref{lem:reduct-trans-groups}   and \fref{lem:reduct-prim-groups}.

  Let us first assume that $G$ is primitive.
  If $G$ has odd order then our statement follows from the Claim in
  the proof of \fref{lem:nulladik}.
  Otherwise $G$ contains an element of order two,
  hence it has a self-paired suborbit \cite[3.2.5]{DM},
  and we can apply \fref{lem:reduct-prim-groups}.

  In case \fref{item:1} and \fref{item:2} of
  \fref{lem:reduct-prim-groups}
  our statement follows.
  In case \fref{item:3} we may assume that $m\ge t$,
  and set $k=\frac nt$. Note that $k\ge\sqrt{n}$ and $t\ge2$.
  We have
  $2(\log(tk))^2 = 2(\log t)^2 +2(\log k)^2 +4\log t\log k\ge
  2(\log t)^2 + 2(\log k)^2 +2\log n\ge
  \big|\Comp_A(X_\alpha)\big|+\big|\Comp_A(Y_\beta)\big|+2\log n\ge
  \big|\Comp_A(G_x)\big|$.
  The inductive step follows.

  The inductive step for transitive imprimitive groups
  (using \fref{lem:reduct-trans-groups})
  is similar, but easier.
\end{proof}

Next we will use \fref{thm:7}
to give a CFSG-free proof of a weaker version of
the key group-theoretic result of Babai \cite{Ba1}.

We will also rely on the following well-known fact.

\begin{prop}\label{prop:weak-ONan-Scott}
The sum of the first $m$ primes is $O(m^2 \log m)$.
\end{prop}

See \cite{Sinha}
for a recent reference concerning sharp bounds.


\vskip 5pt
Here is the promised weaker version of Lemma 8.5 of \cite{Ba1}.

\begin{lem} \label{lem:key}
  Let $G$ be a primitive group of degree $n$. Assume that
  $\phi: G \to \Alt(k)$ is an epimorphism,
  where $k > (\log n)^5$. 
If $k$ is large enough then $\varphi$ is an isomorphism; hence $G\cong \Alt(k)$.
\end{lem}

\begin{proof}

Assume that $N=\ker(\varphi)$ is non-trivial. Then $N$ is transitive
and $G/N\cong\Alt(k)$ contains a cyclic subgroup $C$ whose order is a
product of $c\sqrt{k/\log k}$ different primes (for some absolute
constant $c>0$).
This follows easily using \fref{prop:weak-ONan-Scott}.

  The extension $H$ of $N$ by $C$ is transitive (since $N$ itself is
  transitive) and we have $\big|\Comp_A (H)\big| \ge c\sqrt{k/\log k}$,
  which contradicts \fref{thm:7} for $k$ large enough.
\end{proof}

\begin{rem}
  Babai \cite{Ba1}
  proves the same statement for $k\ge\max\{8,2+\log n\}$
  using Schreier's Hypothesis.
\end{rem}

\begin{rem}
  In January 2017 Helfgott has found an error in the
  ``Split-or-Johnson'' section of Babai's paper \cite{Ba1}.
  This has been corrected in a week by Babai.
  See \cite{He} and \cite{Ba3} for the description of the correction.
  This does not effect the use of the present paper in the analysis of
  Babai's algorithm.
\end{rem}

\vskip 5pt\noindent
{\Large\bf Acknowledgments}
\vskip 5pt\noindent
We thank Laci Babai for finding a gap in the first version of this
proof and his enthusiasm in closing it.
We also thank Colva Roney-Dougal whose thesis 
provided a crucial hint, as to how to close the gap,
and Peter Cameron for drawing our attention to her thesis. The gap was actually closed during a pleasant workshop at the ESI. 
Finally we thank Bandi Szab\'o for his help in writing this paper and
for an essential simplification of the proof.

\vskip 16pt
MTA Alfr\'ed R\'enyi Institute of Mathematics,
Re\'altanoda u. 13--15, H--1053, Budapest, Hungary
\vskip 4pt
{\em E-mail address}: pyber.laszlo@renyi.mta.hu

\end{document}